\documentclass[12pt,a4paper,oneside]{amsart}
\usepackage[a4paper]{geometry}
\usepackage{fouriernc}

\usepackage{enumerate,mathrsfs}

\newtheorem{theorem}{Theorem}
\newtheorem{lemma}{Lemma}
\newtheorem{proposition}{Proposition}
\newtheorem*{corollary}{Corollary}
\newtheorem*{conjecture}{Conjecture}

\newcommand{\vertbar}{\>|\>}
\newcommand{\set}[2]{\ensuremath{\{ #1 \vertbar #2 \}}}
\def\liebrack{\ensuremath{[\,\cdot\, , \cdot\,]}}

\DeclareMathOperator{\ad}{\mathsf{ad}}
\DeclareMathOperator{\Aut}{Aut}
\DeclareMathOperator{\B}{B}
\DeclareMathOperator{\dcobound}{d}
\DeclareMathOperator{\End}{End}
\DeclareMathOperator{\gl}{\mathsf{gl}}
\DeclareMathOperator{\Hom}{Hom}
\DeclareMathOperator{\HomCycl}{HomCycl}
\DeclareMathOperator{\HomLie}{HomLie}
\DeclareMathOperator{\Homol}{H}
\DeclareMathOperator{\HomtwoNilp}{Hom2Nilp}
\DeclareMathOperator{\id}{id}
\DeclareMathOperator{\Ker}{Ker}
\DeclareMathOperator{\QDer}{QDer}
\DeclareMathOperator{\Sl}{\mathsf{sl}}
\DeclareMathOperator{\Z}{Z}

\begin{document}

\title{Hom-Lie structures on Kac--Moody algebras}

\author{Abdenacer Makhlouf}
\address{Universit\'e de Haute-Alsace, Mulhouse, France}
\email{abdenacer.makhlouf@uha.fr}

\author{Pasha Zusmanovich}
\address{University of Ostrava, Ostrava, Czech Republic}
\email{pasha.zusmanovich@osu.cz}
\date{First written April 24, 2018; last minor revision September 15, 2018}
\thanks{J. Algebra \textbf{515} (2018), 278--297; arXiv:1805.00187}

\begin{abstract}
We describe Hom-Lie structures on affine Kac--Moody and related Lie algebras, 
and discuss the question when they form a Jordan algebra.
\end{abstract}

\maketitle

\section*{Introduction}

Hom-Lie algebras (under different names, notably 
``$q$-deformed Witt or Virasoro'') started to appear long time ago in 
the physical literature, in a constant quest for deformed, in that or another
sense, Lie algebra structures bearing a physical significance. The last decade
has witnessed a surge of interest in them, starting with an influential paper 
\cite{hls}.

Recall that a \emph{Hom-Lie algebra} is an anticommutative algebra $L$ with 
multiplication denoted by $\liebrack$, and with a linear map 
$\varphi: L \to L$ such that the following \emph{Hom-Jacobi identity} holds:
\begin{equation}\label{eq-hom-lie}
[[x,y],\varphi(z)] + [[z,x],\varphi(y)] + [[y,z],\varphi(x)] = 0
\end{equation}
for any $x,y,z\in L$. Variations of this notion assume, additionally, that the 
map $\varphi$ is a homomorphism of the algebra $L$, i.e., 
\begin{equation}\label{eq-hom}
\varphi([x,y]) = [\varphi(x),\varphi(y)]
\end{equation}
for any $x,y\in L$, or that $\varphi$ is an isomorphism. Such Hom-Lie algebras 
will be called \emph{multiplicative} or \emph{regular}, respectively.

In this context, the following question arises: to describe all possible 
\emph{Hom-Lie structures} on a given Lie algebra $L$, i.e., all linear maps 
$\varphi: L \to L$ such that the identity (\ref{eq-hom-lie}) holds. (As a more
specific question, one may ask for a description of all \emph{multiplicative}, 
respectively \emph{regular}, \emph{Hom-Lie structures}, i.e. those Hom-Lie
structures which are, additionally, homomorphism, respectively, isomorphism).
Earlier, this task was accomplished for finite-dimensional semisimple Lie 
algebras of characteristic zero (\cite{jl}, \cite{chin2}), for the two-sided 
Witt algebra, for Lie algebras of Cartan type, and for loop algebras (all of 
them are infinite-dimensional Lie algebras of characteristic zero; see 
\cite{xie-liu}). Here we determine Hom-Lie structures on affine Kac--Moody 
algebras. Though the final answer turns out to be not very inspiring (all such structures are 
trivial, in a sense; see Theorem \ref{th-m}), some intermediate computations and
observations, we believe, are of interest.

The content of the paper is as follows. In the auxiliary \S \ref{sec-str} we 
show that the space of Hom-Lie structures on a Lie algebra $L$ carries a natural
structure of an $L$-module. In \S \ref{sec-sim} we prove, in two different ways,
that finite-dimensional simple Lie algebras of classical type do not have 
nontrivial Hom-Lie structures except of the case of $\Sl(2)$. This result is not
new, see \cite{chin2}; our proofs however, are considerably shorter than those 
in \cite{chin2}, do not use case-by-case considerations by types of algebras, 
and do not rely on computer calculations. In \S \ref{sec-tens} we describe 
Hom-Lie structures on Lie algebras represented as tensor products of algebras 
over Koszul dual operads, in the case when one of the operads is commutative, 
and another one is anti-commutative; current Lie algebras come as a particular 
case. In \S \ref{sec-km} we extend these computations to untwisted affine 
Kac--Moody algebras, using their realization as extensions of current Lie 
algebras of a particular type (tensor product of a simple finite-dimensional Lie
algebra with Laurent polynomials). In \S \ref{sec-tw} we treat the case of 
twisted affine Kac--Moody algebras. In \S \ref{sec-jord} we discuss when Hom-Lie
algebra structures on a given Lie algebra form a Jordan algebra, a question raised in \cite{chin2} and \cite{xie-liu}.

\section{Hom-Lie structures as a module}\label{sec-str}

In what follows, the ground field $K$ is assumed arbitrary, of characteristic 
$\ne 2,3$, unless it is stated otherwise. All unadorned $\Hom$'s and tensor 
products are understood over $K$.

Obviously, the set of all Hom-Lie structures on a given Lie algebra $L$ forms
a vector subspace of $\End(L) = \Hom(L,L)$, containing the identity map $\id_L$.

For any natural number $n$, we have the standard $L$-action on the space of 
$n$-linear maps $\Hom(L^{\otimes n},L) \simeq L \otimes (L^*)^{\otimes n}$:
\begin{equation}\label{eq-action}
y \bullet \varphi(x_1, \dots, x_n) = [\varphi(x_1, \dots, x_n),y] 
- \sum_{i=1}^n \varphi(x_1, \dots, x_{i-1}, [x_i,y], x_{i+1}, \dots, x_n) ,
\end{equation}
where $x_1,\dots,x_n,y\in L$, and $\varphi\in \Hom(L^{\otimes n},L)$.

\begin{lemma}\label{lemma-invar}
For any Lie algebra $L$, the space of all Hom-Lie structures on $L$ forms an
$L$-submodule of $\Hom(L,L)$.
\end{lemma}

\begin{proof}
Denote by $J_\varphi(x,y,z)$ the left-hand side of the Hom-Jacobi identity 
(\ref{eq-hom-lie}), considered as an element of $\Hom(L^{\otimes 3},L)$.
The statement of lemma follows from the identity
$$
J_{h\bullet\varphi}(x,y,z) = h\bullet J_\varphi(x,y,z)
$$
which holds for any $x,y,z,h\in L$, and $\varphi\in \Hom(L,L)$ (note that the 
bullet at the left-hand side denotes the action (\ref{eq-action}) on 
$\Hom(L,L)$, i.e. with $n=1$, whereas the bullet at the right-hand side denotes 
the same action on $\Hom(L^{\otimes 3}, L)$, i.e. with $n=3$). This identity is
verified by a straightforward computation.
\end{proof}

The $L$-module of all Hom-Lie structures on a Lie algebra $L$ will be denoted as
$\HomLie(L)$; it always contains a trivial submodule $K\id_L$.

Suppose that $T$ is a torus in a Lie algebra $L$, and 
$L = L_0 \oplus \bigoplus_{\alpha\in R} L_\alpha$ is the root space 
decomposition with respect to the $T$-action. Then, according to 
Lemma \ref{lemma-invar}, this action induces a semisimple action on 
$\HomLie(L)$, with the corresponding root space decomposition
\begin{equation}\label{hom-lie-root}
\HomLie(L) = \bigoplus_{\alpha\in R - R} \HomLie_\alpha(L) ,
\end{equation}
where $R-R$ denotes the set of all differences between two roots in $R$, and
$$
\HomLie_\alpha(L) = 
\set{\varphi\in \HomLie(L)}{\varphi(L_\beta) \subseteq L_{\beta+\alpha} \text{ for any } \beta\in R \cup \{0\}} .
$$

Another related observation is that the space of Hom-Lie structures is stable
with respect to the standard action by conjugation of $\Aut(L)$ on $\End(L)$. 
Indeed, if $\varphi$ is a Hom-Lie structure, and $\alpha$ is an automorphism of
$L$, then writing the Hom-Jacobi identity for the triple 
$\alpha(x)$, $\alpha(y)$, $\alpha(z)$, and then applying to it the automorphism $\alpha^{-1}$, we get
$$
  [[x,y], \alpha^{-1}(\varphi(\alpha(z)))] 
+ [[z,x], \alpha^{-1}(\varphi(\alpha(y)))] 
+ [[y,x], \alpha^{-1}(\varphi(\alpha(y)))] 
= 0
$$
for any $x,y,z \in L$., i.e. $\alpha^{-1} \circ \varphi \circ \alpha$ is also
a Hom-Lie structure on $L$.

\section{Simple Lie algebras of classical type}\label{sec-sim}

Consider the $3$-dimensional simple Lie algebra with the basis $e_-, h, e_+$
and multiplication table
\begin{equation*}
[e_-,h] = -e_-, \quad [e_+,h] = e_+, \quad [e_-,e_+] = h .
\end{equation*}

Over a field of characteristic $\ne 2$, this algebra is isomorphic to $\Sl(2)$.
The computation of Hom-Lie structures on this algebra is straightforward and was
recorded several times in the literature; see, for example, 
\cite[Proposition 7.1]{ms} or \cite[Theorem 3.3(i)]{chin2}. We record it also
here for completeness and further references.

\begin{proposition}\label{prop-sl2}
As an $\Sl(2)$-module, $\HomLie(\Sl(2))$ is isomorphic to the direct sum of the
$5$-di\-men\-si\-o\-nal irreducible module and the one-dimensional trivial 
module.
\end{proposition}

\begin{proof}
The Hom-Lie equation for a linear map $\varphi: \Sl(2) \to \Sl(2)$ is equivalent
to
$$
[h,\varphi(h)] + [e_-,\varphi(e_+)] + [e_+,\varphi(e_-)] = 0.
$$
This removes $3$ parameters from the $9$-dimensional vector space 
$\End(\Sl(2))$, so $\HomLie(\Sl(2))$ is $6$-dimensional. The elements of the
quotient $\HomLie(\Sl(2))/K\id_{\Sl(2)}$, being written as $3\times 3$ matrices
in the basis $\{e_-, h, e_+\}$, can be represented as additive equivalence 
classes, modulo identity matrix, of matrices of the form
$$
\left(
\begin{matrix}
\alpha      & \beta  & \gamma \\
\delta      & 0      & \beta  \\
\varepsilon & \delta & \alpha
\end{matrix}
\right)
$$
where $\alpha,\beta,\gamma,\delta,\varepsilon \in K$. Elementary computations 
show then that, as an $\Sl(2)$-module, this $5$-dimensional space is isomorphic
to the irreducible (highest weight) module. Since representations of $\Sl(2)$
are completely irreducible, $\Hom(\Sl(2))$ decomposes as the direct sum of
this module and the one-dimensional trivial module.
\end{proof}

Till the end of this section, $\mathfrak g$ denotes a finite-dimensional simple
Lie algebra of classical type (this, of course, includes all finite-dimensional
simple Lie algebras over an algebraically closed field of characteristic zero),
not isomorphic to $\Sl(2)$ .

\begin{theorem}\label{th-s}
$\HomLie(\mathfrak g) = K \id_{\mathfrak g}$.
\end{theorem}

As was noted in the introduction, this was already established in \cite{chin2}.
We will provide two alternative proofs of this theorem, one basing on properties
of root systems, and another one utilizing connection with other invariants of
$\mathfrak g$, such as $2$-cocycles and generalized derivations.

\begin{proof}[First proof]
Let $\ell > 1$ be the rank of $\mathfrak g$. Assuming the standard root space 
decomposition $\mathfrak g = H \oplus \bigoplus_{\alpha\in R} Ke_\alpha$ with respect to a 
Cartan subalgebra $H$ of $\mathfrak g$, consider the induced decomposition 
(\ref{hom-lie-root}). Let $\varphi \in \HomLie_\sigma(\mathfrak g)$ for 
$\sigma \ne 0$. 

Suppose first that $\varphi(H) = 0$. Let $\alpha\in R$ be an arbitrary root. 
If $\alpha + \sigma \not\in R$, then $\varphi(e_\alpha) = 0$. 
If $\alpha + \sigma \in R$, then $\alpha$ is linearly independent with $\sigma$.
Pick $h\in H$ such that $\alpha(h) = 0$ and $\sigma(h) \ne 0$. Then the 
Hom-Jacobi identity for triple $h, e_\alpha, e_{-\sigma}$ yields 
$[e_{-\sigma}, \varphi(e_\alpha)] = 0$. As $\varphi(e_\alpha) \in K e_{\alpha+\sigma}$, and 
$[e_{-\sigma}, e_{\alpha+\sigma}] \ne 0$, we have in this case 
$\varphi(e_\alpha) = 0$ as well, and $\varphi$ vanishes.

Suppose now that $\varphi(H) \ne 0$. We have $\varphi(h) = f(h) e_\sigma$ for 
any $h\in H$ and some nonzero linear map $f: \mathfrak g \to K$. Normalize $f$ 
by decomposing $H$ as the direct sum of vector spaces $H = Kh \oplus H^\prime$ 
for a suitable $h\in H$ such that $f(h) = 1$ and $f(H^\prime) = 0$. Then, the 
Hom-Jacobi identity for triples $h,h^\prime,e_\alpha$ and $h,h^\prime,e_{-\sigma}$, where 
$h^\prime\in H^\prime$ and $\alpha\in R$ such that $\sigma + \alpha \in R$, 
implies $\alpha(h^\prime) = 0$ and $\sigma(h^\prime) = 0$, respectively. This, in turn, implies that the set
of roots $\alpha\in R$ such that $\alpha + \sigma\in R$ is empty, an obvious
contradiction. 

Therefore, $\varphi \in \HomLie_0(L)$, i.e. $\varphi$ preserves the standard 
root space decomposition. The rest is easy and follows, more or less, the line 
of reasonings in \cite{jl}. The Hom-Jacobi identity for triple 
$h, h^\prime, e_\alpha$, where $h,h^\prime\in H$ and $\alpha\in R$, implies 
\begin{equation}\label{eq-h}
\alpha(h) \alpha(\varphi(h^\prime)) = \alpha(h^\prime) \alpha(\varphi(h)) .
\end{equation}
Picking in this equality $h^\prime$ such that $\alpha(h^\prime) \ne 0$, we
get that $\varphi(\Ker \alpha) \subseteq \Ker \alpha$. Consequently, 
$$
\varphi(\bigcap_{\alpha \in S} \Ker \alpha) \subseteq \bigcap_{\alpha \in S} \Ker \alpha
$$
for any subset $S \subseteq R$. As for any nonzero $h \in H$ one may choose
$\ell-1$ roots $\alpha_1, \dots, \alpha_{\ell-1}$ such that 
$\bigcap_{1 \le i \le \ell-1} \Ker \alpha_i = Kh$, we have 
$\varphi(h) = \lambda(h) h$ for some map $\lambda: H \to K$. Now the equality
(\ref{eq-h}) reads 
$$
\alpha(h)\alpha(h^\prime)(\lambda(h) - \lambda(h^\prime)) = 0
$$ 
for any $h,h^\prime \in H$ and $\alpha \in R$, what implies that 
$\lambda(\cdot)$ is a constant map: $\lambda(h) = \lambda$ for some $\lambda \in K$. Writing
$\varphi(e_\alpha) = \lambda_\alpha e_\alpha$ for some $\lambda_\alpha \in K$, 
the Hom-Jacobi identity for triple $h\in H$, $e_\alpha, e_\beta$ such that 
$\alpha+\beta \in R$, yields
$$
\alpha(h) (\lambda - \lambda_\beta) + \beta(h) (\lambda - \lambda_\alpha) = 0 .
$$
Picking $h$, $\alpha$ and $\beta$ such that $\alpha(h) = 0$ and 
$\beta(h) \ne 0$, we get $\lambda_\alpha = \lambda$ for any such $\alpha$, and
varying $h$, we get $\lambda_\alpha = \lambda$ for any $\alpha \in R$. 
Therefore, $\varphi = \lambda \id_L$.
\end{proof}

To provide a second proof of Theorem \ref{th-s}, we need a series of definitions
and lemmas, some of them could be of independent interest.

For a Lie algebra $L$, a bilinear map: $\varphi: L \times L \to K$ satisfying
the $2$-cocycle equation
$$
\varphi([x,y],z) + \varphi([z,x],y) + \varphi([y,z],x) = 0
$$
for any $x,y,z \in L$, is called an \emph{asymmetric $2$-cocycle} (note that we
do not require $\varphi$ to be skew-symmetric, like in the classical definition
of the Chevalley--Eilenberg cohomology, or symmetric, like in \cite{comm2}). 
The vector space of all asymmetric $2$-cocycles on $L$ is denoted by 
$\mathcal Z^2(L)$.

\begin{lemma}\label{lemma-t}
Let $L$ be a Lie algebra with a bilinear symmetric invariant form 
$\langle \cdot,\cdot \rangle$, $t\in L$, and $\varphi$ a Hom-Lie structure on 
$L$. Then the map $f_t: L \times L \to K$ defined by 
$f_t(x,y) = \langle \varphi(y),[x,t] \rangle$ for $x,y\in L$, is an asymmetric
$2$-cocycle on $L$.
\end{lemma}

\begin{proof}
Applying $\langle \cdot, t \rangle$ to the both sides of the Hom-Jacobi 
identity, and using invariance of $\langle \cdot,\cdot \rangle$, we get the 
$2$-cocycle equation for $f_t$.
\end{proof}

Now, letting $L$ be a Lie algebra and $M$ an $L$-module, denote by $\QDer(L,M)$ 
the space of linear maps $D: L \to M$ such that there exists another linear map 
$F_D: L \to M$ satisfying the condition 
$$
D([x,y]) = y \bullet F_D(x) - x \bullet F_D(y)
$$ 
for any $x,y\in L$, where $\bullet$ denotes the module action. Such maps are 
called \textit{quasiderivations} (or sometimes \emph{generalized derivations}) 
in the literature. In the special case where $F_D = \delta D$ for some fixed 
$\delta \in K$, they are called \emph{$\delta$-derivations}, and in the case 
$\delta = -1$, \emph{antiderivations}. All these generalizations of derivations
were studied extensively in the literature in the case $M=L$, the adjoint 
module; see, for example, \cite{filippov-1} and \cite{leger-luks}, and references therein.

Let $\B(L)$ denote the space of bilinear maps $\varphi: L \times L \to K$ such 
that
$$
\varphi([x,y],z) = \varphi([z,x],y)
$$ 
for any $x,y,z\in L$.

\begin{lemma}\label{lemma-seq}
For any Lie algebra $L$, the following sequence is exact:
\begin{equation*}
0 \to \mathcal Z^2(L) \overset{u}\to \QDer(L,L^*) \overset{v}\to \B(L) ,
\end{equation*}
where the maps $u$ and $v$ are defined as follows: for 
$\varphi\in \mathcal Z^2(L)$, 
$$
(u(\varphi)(x))(y) = \varphi(x,y) ,
$$
with 
$$
(F_{u(\varphi)}(x))(y) = -\varphi(y,x) ,
$$
and for $D\in \QDer(L,L^*)$, 
$$
v(D)(x,y) = D(x)(y) + F_D(y)(x)
$$
for any $x,y\in L$.
\end{lemma}

Being restricted to skew-symmetric $2$-cocycles, this exact sequence reduces
to a well-known one, which goes back to the classic works of Koszul and 
Hoch\-schild--Serre, connecting the second cohomology of a Lie algebra $L$
with coefficients in the trivial module $\Homol^2(L,K)$, the first cohomology 
with coefficients in the coadjoint module $\Homol^1(L,L^*)$, and the space of 
symmetric invariant bilinear forms on $L$. Being restricted to symmetric 
$2$-cocycles, this exact sequence reduces to a ``symmetric'' analog of the 
latter, relating the so-called commutative $2$-cocycles and antiderivations (see
\cite[\S 1]{comm2}). The proof of Lemma \ref{lemma-seq} is straightforward, 
repeating almost verbatim the proof of the latter symmetric analog, and is 
omitted here.

Now we apply these lemmas to the specific case of simple Lie algebras of 
classical type with the Killing form $\langle \cdot,\cdot \rangle$.

\begin{lemma}\label{lemma-cl}
$\mathcal Z^2(\mathfrak g) = \B^2(\mathfrak g,K)$. In particular, any asymmetric
$2$-cocycle on $\mathfrak g$ is skew-symmetric.
\end{lemma}

(Here $\B^2$ denotes the usual space of Chevalley--Eilenberg $2$-coboundaries).

\begin{proof}
By \cite[Corollary 4.14]{leger-luks}, 
$\QDer(\mathfrak g, \mathfrak g) = \ad(\mathfrak g) \oplus K \id_{\mathfrak g}$.
But since $\mathfrak g^* \simeq \mathfrak g$ as $\mathfrak g$-modules, by 
Lemma \ref{lemma-seq} any asymmetric $2$-cocycle $\varphi$ on $\mathfrak g$ has
the form $\varphi(x,y) = \langle [x,t] + \lambda x , y \rangle$, where 
$x,y\in \mathfrak g$, and $t$ and $\lambda$ are some fixed elements of 
$\mathfrak g$ and $K$, respectively. An elementary computation shows that $\lambda = 0$, and
hence $\varphi$ is skew-symmetric and coincides with a $2$-coboundary on 
$\mathfrak g$.
\end{proof}

\begin{proof}[Second proof of Theorem \ref{th-s}]
Let $\varphi$ be a Hom-Lie structure on $\mathfrak g$. Combining 
Lemmas \ref{lemma-t} and \ref{lemma-cl}, we get
$$
\langle \varphi(y),[x,t] \rangle + \langle \varphi(x),[y,t] \rangle = 0
$$
for any $x,y,t \in L$. Using invariance and nondegeneracy of 
$\langle \cdot,\cdot \rangle$, this equality is equivalent to
$$
[\varphi(y),x] + [\varphi(x),y] = 0 .
$$
Then by \cite[Theorem 5.23]{leger-luks}, $\varphi$ is a scalar multiple of 
$\id_{\mathfrak g}$.
\end{proof}

Computations of Hom-Lie structures on both finite- and infinite-dimensional
simple Lie algebras in characteristic zero, suggest the following 

\begin{conjecture}
If a simple finite-dimensional Lie algebra admits a nontrivial Hom-Lie 
structure, then it is isomorphic either to a $3$-dimensional simple algebra,
or, in the case of positive characteristic, to the Zassenhaus algebra.
\end{conjecture}

If true, this conjecture will add to numerous characterizations of $\Sl(2)$ and
the Zassenhaus algebra among simple Lie algebras, see \cite[\S 2]{comm2} and 
references therein.

\section{Tensor product Lie algebras}\label{sec-tens}

In \cite[\S 4]{vavilov-fest}, Hom-Lie structures on current Lie algebras, i.e. 
Lie algebras of the form $L \otimes A$, where $L$ is a Lie algebra, and $A$ is 
an associative commutative algebra, subject to multiplication
$$
[x \otimes a , y \otimes b] = [x,y] \otimes ab ,
$$
where $x,y \in L$, $a,b \in A$, were described in terms of tensor factors $L$ 
and $A$ (subject to technical assumption that one of $L$, $A$ is 
finite-dimensional, and that $A$ contains a unit):
\begin{equation}\label{eq-la}
\HomLie (L\otimes A) \simeq \HomLie(L) \otimes A + 
\set{\varphi\in \End(L)}{[[L,L],\varphi(L)] = 0} \otimes \End(A)
\end{equation}
(a particular case of Lie algebras of the form 
$\mathfrak g \otimes \mathbb C[t,t^{-1}]$ was done, via direct computations, 
also in \cite{chin2}).

Current Lie algebras are particular case of the following construction, known
from the operadic theory: if $A$ is an algebra over a binary quadratic operad
$\mathscr P$, $B$ an algebra over the operad $\mathscr P^!$ Koszul dual to 
$\mathscr P$, then the tensor product $A \otimes B$ is a Lie algebra subject to
multiplication defined by
$$
[a \otimes b, a^\prime \otimes b^\prime] = 
aa^\prime \otimes bb^\prime - a^\prime a \otimes b^\prime b ,
$$
where $a,a^\prime \in A$, $b,b^\prime \in B$. This Lie algebra will be denoted
by $(A \otimes B)^{(-)}$.

An interesting question is to describe $\HomLie\big((A \otimes B)^{(-)}\big)$ in
terms of some invariants of $A$ and $B$ in the case of arbitrary pair of Koszul
dual binary quadratic operads, similarly to (\ref{eq-la}), which corresponds to
the pair of operads (Lie, associative commutative). If such description would be
possible, it will give at once description of Hom-Lie structures on many 
infinite-dimensional Lie algebras of physical origin (Poisson brackets of 
hydrodynamic type, Heisenberg--Virasoro, Schr\"odinger--Virasoro, etc.), as they can be represented as Lie algebras of 
the form $(A \otimes B)^{(-)}$ for the pair of operads 
(left Novikov, right Novikov) -- see \cite[\S 5]{vavilov-fest} and references 
therein, as well as for Lie algebras $\gl_n(A)$ of matrices over an associative
algebra $A$ -- what corresponds to the self-dual pair of operads 
(associative, associative).

However, we are able to answer this question only in a particular case, where
one of the operads is commutative, and another one is anticommutative. To 
formulate the result, we need to extend the notion of Hom-Lie structure. 
Recently, it became popular to study Hom-algebras attached to various varieties
of (nonassociative) algebras, by ``twisting'' identities defining the 
respective varieties, similarly how Hom-Jacobi identity is a twisted variant of
the ordinary Jacobi identity (see, for example, \cite{paradigm} and references
therein). Following this trend, we will be interested, besides Hom-Lie algebras,
in Hom-cyclic algebras and Hom-2-nilpotent algebras, corresponding to the 
variety of cyclic algebras defined by the identity
$$
(ab)c = (ca)b ,
$$
and the variety of $2$-nilpotent algebras defined by the identity
$$
(ab)c = 0 .
$$

The corresponding ``twisted'' identities are Hom-cyclic:
\begin{equation}\label{eq-hc}
(ab)\varphi(c) = (ca)\varphi(b) ,
\end{equation}
and Hom-$2$-nilpotent:
\begin{equation}\label{eq-2n}
(ab)\varphi(c) = 0 .
\end{equation}

However, we are shifting the emphasis a bit and consider the corresponding 
Hom-structures on an arbitrary (nonassociative) algebra $A$. The vector spaces 
of the corresponding Hom-structures, that is, all linear maps $\varphi: A \to A$
satisfying the Hom-Jacobi identity
$$
(ab)\varphi(c) + (ca)\varphi(b) + (bc)\varphi(a) = 0 ,
$$
or the Hom-cyclic identity (\ref{eq-hc}), or the Hom-$2$-nilpotent identity 
(\ref{eq-2n}), will be denoted by $\HomLie(A)$, $\HomCycl(A)$, or 
$\HomtwoNilp(A)$, respectively.

\begin{theorem}\label{th-oper}
Let $\mathscr P$ be a commutative binary quadratic operad, $A$ an algebra over 
$\mathscr P$, $B$ an algebra over $\mathscr P^!$, and one of $A$, $B$ is 
finite-dimensional. Then
\begin{align}\label{eq-ab}
\HomLie\big((A \otimes B)^{(-)}\big) \>\simeq\> 
    &\HomLie(A) \otimes \HomCycl(B) + \HomtwoNilp(A) \otimes \End(B) \\
+\> &\HomCycl(A) \otimes \HomLie(B) + \End(A) \otimes \HomtwoNilp(B) . \notag
\end{align}
\end{theorem}

\begin{proof}
The proof follows closely the case of current Lie algebras treated in 
\cite[Theorem 4]{vavilov-fest}. An arbitrary linear map 
$\Phi: A \otimes B \to A \otimes B$ can be written in the form
\begin{equation*}
\Phi(x \otimes a) = \sum_{i\in I} \alpha_i(x) \otimes \beta_i(a) ,
\end{equation*}
where $x \in A$, $a \in B$, and $\alpha_i: A \to A$, $\beta_i: B \to B$ is a 
(finite) family of linear maps, indexed by a set $I$. Note that since $A$ is
commutative, $B$ is anticommutative, and the Hom-Jacobi identity for $\Phi$ 
takes the form
\begin{equation}\label{eq-hj}
\sum_{i\in I} \Big(
  (xy)\alpha_i(z) \otimes (ab)\beta_i(c)
+ (zx)\alpha_i(y) \otimes (ca)\beta_i(b) 
+ (yz)\alpha_i(x) \otimes (bc)\beta_i(a)
\Big) = 0
\end{equation}
for any $x,y,z \in A$, $a,b,c \in B$.

Cyclically permuting in this equality $x,y,z$, and summing up the obtained $3$
equalities, we get:
\begin{equation}\label{eq-1}
\sum_{i\in I} 
\Big((xy)\alpha_i(z) + (zx)\alpha_i(y) + (yz)\alpha_i(x) \Big) 
\otimes 
\Big((ab)\beta_i(c) + (ca)\beta_i(b) + (bc)\beta_i(a)\Big) = 0 .
\end{equation}

Skew-symmetrizing the equality (\ref{eq-hj}) with respect to $x,y$, we get:
\begin{equation}\label{eq-2}
\sum_{i\in I} 
\Big((yz)\alpha_i(x) - (zx)\alpha_i(y)\Big) 
\otimes 
\Big((bc)\beta_i(a) - (ca)\beta_i(b)\Big) = 0 .
\end{equation}

Applying \cite[Lemma 1.1]{low} to the last two equalities, we may assume that 
the indexing set is partitioned into the four subsets: 
$I = I_{11} \cup I_{12} \cup I_{21} \cup I_{22}$, each of them is characterized
by the condition of vanishing of one of the tensor factors in (\ref{eq-1}),
and of one of the tensor factors in (\ref{eq-2}). Noting that these vanishing
conditions are exactly the Hom-Jacobi identities in the case of (\ref{eq-1}), 
and Hom-cyclic identities in the case of (\ref{eq-2}), we may write:
$$
\begin{array}{lll}
\alpha_i \in \HomLie(A) ; & \alpha_i \in \HomCycl(A) & \text{ for } i \in I_{11}
\\
\alpha_i \in \HomLie(A) ; & \beta_i \in \HomCycl(B) & \text{ for } i \in I_{12}
\\
\alpha_i \in \HomCycl(A) ; & \beta_i \in \HomLie(B) & \text{ for } i \in I_{21}
\\
\beta_i \in \HomLie(B) ; & \beta_i \in \HomCycl(B)  & \text{ for } i \in I_{22} 
.
\end{array}
$$

But for any algebra $A$, 
$$
\HomLie(A) \cap \HomCycl(A) = \HomtwoNilp(A) .
$$
It is easy to see then that each of the decomposable maps 
$\alpha_i \otimes \beta_i$ for 
$i \in I_{11}$, $i\in I_{12}$, $i\in I_{21}$, $i\in I_{22}$ satisfies the 
Hom-Jacobi identity (\ref{eq-hj}), and the statement of the theorem follows.
\end{proof}

The isomorphism (\ref{eq-la}) is a particular case of (\ref{eq-ab}): indeed, in
the class of associative commutative algebras with unit, Hom-cyclic structures
are exactly multiplications by an element of the algebra, and Hom-$2$-nilpotent
structures vanish.

\section{Untwisted Kac--Moody algebras}\label{sec-km}

We employ the well-known realization of affine Kac--Moody algebras as untwisted 
and twisted extensions of current algebras of a particular type. In this section
we will treat the untwisted case.

Let $\mathfrak g$ be a finite-dimensional simple complex Lie algebra, and 
$\mathbb C[t,t^{-1}]$ the algebra of Laurent polynomials. Then an untwisted
affine Kac--Moody algebra is the universal central extension (via the famous 
``Kac--Moody $2$-cocycle'') of the current Lie algebra 
$\mathfrak g \otimes \mathbb C[t,t^{-1}]$ extended by the Euler derivation 
$\id_{\mathfrak g} \otimes t \frac{\dcobound}{\dcobound t}$. Accordingly, 
starting with a current Lie algebra, we will see how Hom-Lie structures change 
first by extending the current algebra by derivation, and then by a central 
element. Our first task is to determine Hom-Lie structures on Lie algebras of 
the form 
$$
\big(\mathfrak g \otimes \mathbb C[t,t^{-1}]\big) \oplus
\big(\mathbb C \id_{\mathfrak g} \otimes t \frac{\dcobound}{\dcobound t}\big)
$$
(the direct sum is understood here in the category of vector spaces; in the 
category of Lie algebras, this is a semidirect sum). We will treat a slightly 
more general case: semidirect sums of the form 
$\big(L \otimes A\big) \oplus \big(K\id_L \otimes D\big)$, where $D$ acts on 
$L \otimes A$ by a derivation of the second tensor factor $A$.

\begin{proposition}\label{th-extd}
Let $L$ be a Lie algebra satisfying the following conditions:
\begin{enumerate}[\upshape(i)]
\item $L$ does not have abelian subalgebras of codimension $1$;
\item the center of $L$ is zero;
\item $\set{x\in L}{[[L,L],x] = 0} = 0$.
\end{enumerate}
Let $A$ be an associative commutative algebra with unit and without zero 
divisors, one of $L$, $A$ is finite-dimensional, and $D$ a derivation of $A$ 
such that $\dim D(A) > 1$. Then 
$$
\HomLie \Big(\big(L \otimes A\big) \oplus \big(K\id_L \otimes D\big)\Big) =
K \id_{(L \otimes A) \oplus (K\id_L \otimes D)} .
$$
\end{proposition}

\begin{proof}
The proof is based on the proof of \cite[Theorem 4]{vavilov-fest} (or the proof
of a more general Theorem \ref{th-oper}).

The restriction of a Hom-Lie structure $\Phi$ on 
$\big(L\otimes A\big) \oplus \big(K\id_L \otimes D\big)$ to $L \otimes A$ can 
be written as
\begin{equation*}
\Phi(x \otimes a) = \sum_{i\in I} \varphi_i(x) \otimes \alpha_i(a) + 
\sum_{i \in J} \lambda_i(x) \beta_i(a) \id_L \otimes D ,
\end{equation*}
where $x \in L$ and $a \in A$, for some linear maps $\varphi_i: L \to L$, 
$\alpha_i: A \to A$, $\lambda_i: L \to K$, $\beta_i: A \to K$.

The Hom-Jacobi identity for triple $x \otimes a, y \otimes b, z \otimes c$ then
reads
\begin{multline*}
\sum_{i\in I} 
\Big(
  [[x,y],\varphi_i(z)] \otimes ab\alpha_i(c)
+ [[z,x],\varphi_i(y)] \otimes ca\alpha_i(b) 
+ [[y,z],\varphi_i(x)] \otimes bc\alpha_i(a)
\Big)  
\\
+ \sum_{i\in J} 
\Big(
  \lambda_i(z)[x,y] \otimes \beta_i(c)D(ab)
+ \lambda_i(y)[z,x] \otimes \beta_i(b)D(ca)
+ \lambda_i(x)[y,z] \otimes \beta_i(a)D(bc)
\Big)
= 0
\end{multline*}
for any $x,y,z\in L$, $a,b,c\in A$.

Cyclically permuting in this equality $x,y,z$, and summing up the obtained $3$ 
equalities, we get:
\begin{multline}\label{eq-c}
\sum_{i\in I} 
\Big([[x,y],\varphi_i(z)] + [[z,x],\varphi_i(y)] + [[y,z],\varphi_i(x)] \Big)
\otimes 
\Big(ab\alpha_i(c) + ca\alpha_i(b) + bc\alpha_i(a)\Big)
\\
+ \sum_{i\in J} 
\Big(\lambda_i(z)[x,y] + \lambda_i(y)[z,x] + \lambda_i(x)[y,z]\Big) \otimes  
\Big(\beta_i(c)D(ab) + \beta_i(b)D(ca) + \beta_i(a)D(bc) \Big)
= 0 .
\end{multline}

In the sum at the left-hand side of the last equality, either all the first 
tensor factors vanish, or a nontrivial linear combination of the second tensor 
factors vanishes. If the second possibility holds, then
\begin{equation*}
  ab\alpha(c) + ca\alpha(b) + bc\alpha(a) 
+ \beta(c)D(ab) + \beta(b)D(ca) + \beta(a)D(bc) = 0
\end{equation*}
for any $a,b,c\in A$, where $\alpha$ is a linear combination of $\alpha_i$'s and
$\beta$ is a linear combination of $\beta_i$'s, at least one of them is 
nontrivial.

In the last equality, expand the expressions containing $D$ of the product of
two elements via the Leibniz formula, and substitute $c=1$:
\begin{equation}\label{eq-abc}
  ab\alpha(1) + a\alpha(b) + b\alpha(a) 
+ \Big(\beta(b)1 + \beta(1)b\Big)D(a) 
+ \Big(\beta(a)1 + \beta(1)a\Big)D(b) 
= 0 .
\end{equation}

Substituting here further $a=b=1$, we get $\alpha(1) = 0$. Substituting then 
$b=1$, we get $\alpha = -2\beta(1)D$, and substituting this back to 
(\ref{eq-abc}), we get
\begin{equation}\label{eq-ab1}
\Big(\beta(b)1 - \beta(1)b\Big)D(a) + \Big(\beta(a)1 - \beta(1)a\Big)D(b) = 0 .
\end{equation}

Assuming in the last equality $a=b$, we get
$$
\Big(\beta(a)1 - \beta(1)a\Big)D(a) = 0 .
$$

Since $A$ is without zero divisors, $\beta(a)1 = \beta(1)a$ for any 
$a \notin \Ker D$. If $\beta(1) \ne 0$, this means that any $a \notin \Ker D$ is
a scalar multiple of $1$, a contradiction. Hence $\beta(1) = 0$, $\alpha = 0$, 
and the equality (\ref{eq-ab1}) takes the form
$$
\beta(b)D(a) + \beta(a)D(b) = 0 .
$$

If there is $b\in A$ such that $\beta(b) \ne 0$, then for any $a\in A$, $D(a)$ 
is a scalar multiple of $D(b)$, a contradiction with the condition 
$\dim D(A) > 1$. Hence both $\alpha$ and $\beta$ vanish, again a contradiction.

Consequently, each first tensor factor in the sums at the left-hand side of 
(\ref{eq-c}) vanishes; that is, $\varphi_i \in \HomLie(L)$ for any $i \in I$, 
and
\begin{equation*}
\lambda_i(z)[x,y] + \lambda_i(y)[z,x] + \lambda_i(x)[y,z] = 0
\end{equation*}
for any $x,y,z\in L$ and $i\in J$. Assuming in the latter equality 
$x,y \in \Ker \lambda_i$, we get 
$\lambda_i(L)[\Ker \lambda_i, \Ker \lambda_i] = 0$. If $\lambda_i \ne 0$, then
$\Ker \lambda_i$ is of codimension $1$ in $L$, so it cannot be an abelian 
subalgebra, i.e. $[\Ker \lambda_i, \Ker \lambda_i]$ is not zero, and 
$\lambda_i = 0$, a contradiction. 

Thus $\lambda_i = 0$ for any $i \in J$, 
$\Phi(L \otimes A) \subseteq L \otimes A$, and the restriction of $\Phi$ to
$L \otimes A$ is a Hom-Lie structure on $L \otimes A$. Note that the condition
(iii) on $L$ in the statement of the proposition implies $\HomtwoNilp(L) = 0$, 
and thus by \cite[Theorem 4]{vavilov-fest} (formula (\ref{eq-la}) here; or by a
more general Theorem \ref{th-oper}) we have
$$
\Phi (x \otimes a) = \sum_{i\in I} \varphi_i(x) \otimes au_i ,
$$
where $x\in L$, $a\in A$, for suitable $\varphi_i \in \HomLie(L)$ and 
$u_i \in A$. Write
\begin{equation*}
\Phi(\id_L \otimes D) = \sum_{i\in K} t_i \otimes v_i + \lambda \id_L \otimes D
\end{equation*}
for some linearly independent elements $t_i \in L$, linearly independent 
elements $v_i \in A$, and $\lambda \in K$. Then the Hom-Jacobi 
identity for triple $x \otimes a$, $y \otimes b$, $\id_L \otimes D$ is 
equivalent to
\begin{equation}\label{eq-da}
  \sum_{i\in K} [[x,y],t_i] \otimes abv_i 
+ \lambda[x,y] \otimes D(ab)
+ \sum_{i\in I} 
\Big([y,\varphi_i(x)] \otimes aD(b)u_i - [x,\varphi_i(y)] \otimes D(a)bu_i\Big)
= 0
\end{equation}
for any $x,y\in L$ and $a,b\in A$.

Substituting here $a=b=1$ yields
$$
\sum_{i\in K} [[x,y],t_i] \otimes v_i = 0 .
$$

Since $v_i$'s are linearly independent, $[[x,y],t_i] = 0$ and hence $t_i = 0$ 
for any $i\in K$. Taking in the remainder of equality (\ref{eq-da}) $b=1$, we 
get:
\begin{equation*}
\lambda[x,y] \otimes D(a) - \sum_{i\in I} [x,\varphi_i(y)] \otimes D(a)u_i
= 0 .
\end{equation*}

Consider the quotient of the last equality in the quotient vector space
$\End(L)/K\id_L \otimes \End(A)$:
\begin{equation*}
\sum_{i\in I} [x,\overline\varphi_i(y)] \otimes D(a)u_i = 0 ,
\end{equation*}
where $\overline\varphi_i$ is the equivalence class of $\varphi_i$. Since $A$ is
without zero divisors, the vanishing of the second tensor factors here implies 
$u_i = 0$, hence all the first tensor factor vanish. Since the center of $L$ is
zero, this implies $\overline\varphi_i = 0$, i.e. $\varphi_i$ is a scalar 
multiple of $\id_L$ for each $i\in I$. Consequently, 
$\Phi(x \otimes a) = x \otimes au$ for some $u\in A$, and the Hom-Jacobi 
identity (\ref{eq-da}) reduces to
$$
[x,y] \otimes D(ab) (\lambda 1 - u) = 0 .
$$
Then again, since $A$ is without zero divisors, $u = \lambda 1$, and $\Phi$ is
a scalar multiple of $\id_{(L \otimes A) \oplus (\id_L \otimes D)}$.
\end{proof}

Now we will use this result to determine Hom-Lie structures on an affine 
untwisted Kac--Moody algebra
$$
\big(\mathfrak g \otimes \mathbb C[t,t^{-1}]\big) \oplus
\big(\mathbb C \id_{\mathfrak g} \otimes t \frac{\dcobound}{\dcobound t}\big)
\oplus \mathbb Cz ,
$$
where $z$ is the central element, and multiplication on 
$\mathfrak g \otimes \mathbb C[t,t^{-1}]$ is twisted by the Kac--Moody 
$2$-cocycle, whose concrete form is immaterial for our purposes here.

As far as Hom-Lie structures are concerned, central extensions are dealt with 
easily. First, note that for any Lie algebra $\mathscr L$ with a nonzero center
$\Z(\mathscr L)$, any linear map $\mathscr L \to \Z(\mathscr L)$ is a Hom-Lie 
structure on $\mathscr L$. Such Hom-Lie structures will be called 
\emph{central}.

Now let $\widehat{\mathscr L}$ be a one-dimensional central extension of a Lie
algebra $\mathscr L$ with the Lie bracket $\liebrack$. Write 
$\widehat{\mathscr  L}$ as the vector space direct sum $\mathscr L \oplus Kz$, 
where $z$ is the central element, with multiplication
$$
\{x,y\} = [x,y] + \xi(x,y)z
$$
where $x,y \in \mathscr L$, and $\xi$ is a $2$-cocycle on $\mathscr L$. Writing
a linear map $\varphi: \widehat{\mathscr L} \to \widehat{\mathscr L}$ as 
$\varphi(x) = \psi(x) + \lambda(x)z$ for $x\in \mathscr L$, and 
$\varphi(z) = s + \mu z$, where $\psi \in \End(\mathscr L)$, 
$\lambda \in \mathscr L^*$, $s \in \mathscr L$, and $\mu \in K$, the condition 
that $\varphi$ is a Hom-Lie structure in $\widehat{\mathscr L}$ is equivalent to
the following: $\psi \in \HomLie(\mathscr L)$ with an additional condition
\begin{equation}\label{eq-2c}
\xi([x,y],\psi(t)) + \xi([t,x],\psi(y)) + \xi([y,t],\psi(x))  = 0
\end{equation}
for any $x,y,t \in \mathscr L$, $[[\mathscr L,\mathscr L],s] = 0$, and 
$\xi([\mathscr L,\mathscr L],s) = 0$. (Not surprisingly, the ``Hom-$2$-cocycle''
equation (\ref{eq-2c}) appears in description of central extensions in the category of Hom-Lie algebras, see 
\cite[\S 2.4]{hls}).

\begin{proposition}\label{prop-km}
The space of Hom-Lie structures on an affine untwisted Kac--Moody algebra is 
linearly spanned by central Hom-Lie structures and the identity map.
\end{proposition}

\begin{proof}
Keeping notation of the discussion above, in our concrete situation 
$$
\mathscr L = \big(\mathfrak g \otimes \mathbb C[t,t^{-1}]\big) \oplus
\big(\mathbb C \id_{\mathfrak g} \otimes t \frac{\dcobound}{\dcobound t}\big) .
$$
By Proposition \ref{th-extd} any Hom-Lie structure $\psi$ on this algebra is a 
scalar multiple of the identity map, and the Hom-$2$-cocycle equation 
(\ref{eq-2c}) is satisfied trivially. Further, 
$[\mathscr L,\mathscr L] = \mathfrak g \otimes \mathbb C[t,t^{-1}]$, and a
straightforward computation shows that there are no nonzero elements 
$s \in \mathscr L$ such that $[[\mathscr L,\mathscr L],s] = 0$. Thus any Hom-Lie
structure on $\widehat{\mathscr L}$ is a linear combination of a scalar 
multiple of $\id_{\widehat{\mathscr L}}$, and a linear map with values in $Kz$.
\end{proof}

\section{Twisted Kac--Moody algebras}\label{sec-tw}

Let 
$\mathfrak g = \bigoplus_{\overline i\in \mathbb Z / n\mathbb Z} \mathfrak g_{\overline i}$ 
be a $\mathbb Z / n\mathbb Z$-grading of a finite-dimensional simple complex 
Lie algebra $\mathfrak g$ (in practice, $n=2$ or $3$). A twisted affine 
Kac--Moody algebra is an extension of the twisted current Lie algebra 
\begin{equation}\label{eq-tw}
\mathscr L(\mathfrak g,n) = 
\bigoplus_{i\in \mathbb Z} \mathfrak g_{i(mod\, n)} \otimes t^i
\end{equation}
by the Euler derivation, and by the central element. 

Like in the untwisted case, we consider first the extension of the twisted 
current Lie algebra by the Euler derivation:
\begin{equation*}
\mathscr L = 
\mathscr L(\mathfrak g,n) \oplus 
\big(\mathbb C\id_{\mathfrak g} \otimes\> t \frac{\dcobound}{\dcobound t}\big) .
\end{equation*}

\begin{proposition}\label{prop-tw}
Let $\mathscr L$ be as above. Then 
$\HomLie(\mathscr L) = \mathbb C\id_{\mathscr L}$.
\end{proposition}

\begin{proof}
The Euler derivation $\id_{\mathfrak g} \otimes t \frac{\dcobound}{\dcobound t}$
acts semisimply on $\mathscr L$, with root spaces being exactly the homogeneous
components in the grading (\ref{eq-tw}) (the derivation itself, of course, goes
to the zero component). According to \S \ref{sec-str}, we may consider only 
Hom-Lie structures shifting this grading by some integer. So let $\Phi$ be a 
Hom-Lie structure on $\mathscr L$ shifting the grading (\ref{eq-tw}) by an 
integer $\kappa$. We have then for any $i \in \mathbb Z$:
$$
\Phi(x \otimes t^i) = \varphi(x) \otimes t^{i+\kappa} + 
\lambda(x)\id_{\mathfrak g} \otimes t \frac{\dcobound}{\dcobound t} ,
$$
where $x \in \mathfrak g_{i(mod\,n)}$, $\varphi: \mathfrak g \to \mathfrak g$ is
a linear map such that 
$\varphi(\mathfrak g_{i(mod\,n)}) \subseteq \mathfrak g_{(i+\kappa)(mod\,n)}$,
and $\lambda: \mathfrak g \to \mathbb C$ is a linear map vanishing on all
homogeneous components except of $\mathfrak g_{-\kappa(mod\,n)}$. Also,
\begin{align*}
&\Phi(\id_{\mathfrak g} \otimes t \frac{\dcobound}{\dcobound t}) = 
u \otimes t^\kappa
\\
\intertext{if $\kappa \ne 0$, and}
&\Phi(\id_{\mathfrak g} \otimes t \frac{\dcobound}{\dcobound t}) = 
u \otimes t^\kappa + 
\lambda \id_{\mathfrak g} \otimes t \frac{\dcobound}{\dcobound t}
\end{align*}
if $\kappa = 0$, where $\lambda \in \mathbb C$, and, in both cases, 
$u \in \mathfrak g_{\kappa(mod\,n)}$.

The Hom-Jacobi identity for triple $x \otimes t^i$, $y \otimes t^j$,
$z \otimes t^k$ yields
\begin{multline}\label{eq-m}
\Big([[x,y],\varphi(z)] + [[z,x],\varphi(y)] + [[y,z],\varphi(x)]\Big)
\otimes t^{i+j+k+\kappa} 
\\
+ (i+j)\lambda(z)[x,y] \otimes t^{i+j}
+ (k+i)\lambda(y)[z,x] \otimes t^{k+i}
+ (j+k)\lambda(x)[y,z] \otimes t^{j+k}
= 0
\end{multline}
for any $x \in \mathfrak g_{i(mod\,n)}$, $y \in \mathfrak g_{j(mod\,n)}$, 
$z \in \mathfrak g_{k(mod\,n)}$.

The Hom-Jacobi identity for triple $x \otimes 1$, $y \otimes 1$, 
$\id_{\mathfrak g} \otimes t \frac{\dcobound}{\dcobound t}$ yields 
$[[x,y],u] = 0$ for any $x,y \in \mathfrak g_{0(mod\,n)}$, and hence $u=0$. 

Assume first $\kappa \ne 0$. The Hom-Jacobi identity for triple 
$x \otimes t^i$, where $i \ne 0$, $y \otimes 1$, 
$\id_{\mathfrak g} \otimes t \frac{\dcobound}{\dcobound t}$, yields 
$[x,\varphi(y)] = 0$ for any $x \in \mathfrak g_{i(mod\,n)}$ and 
$y \in \mathfrak g_{0(mod\,n)}$, what implies 
$\varphi (\mathfrak g_{0(mod\,n)}) = 0$. Setting in (\ref{eq-m}) $i=j=0$, we get
the Hom-Jacobi identity for any triple of elements 
$x,y \in \mathfrak g_{0(mod\,n)}$ and $z \in \mathfrak g$. This implies 
$[[x,y],\varphi(z)] = 0$, and hence $\varphi$, and thus $\Phi$, is identically
zero. 

Consequently, we may assume $\kappa = 0$ (i.e., $\Phi$ preserves the grading of
$\mathscr L$). Again, the Hom-Jacobi identity for triple 
$x \otimes t^i$, where $i\ne 0$, $y \otimes 1$, 
$\id_{\mathfrak g} \otimes t \frac{\dcobound}{\dcobound t}$, yields
\begin{equation}\label{eq-i}
\lambda [x,y] - [x,\varphi(y)] - i \lambda(y) x = 0
\end{equation}
for any $x \in \mathfrak g_{i(mod\,n)}$ and $y \in \mathfrak g_{0(mod\,n)}$. 
Since $\lambda(\cdot)$ vanishes on a subspace $\mathfrak g_{0(mod\,n)}^\prime$ 
of $\mathfrak g_{0(mod\,n)}$ of codimension $1$, the last equality implies that 
$[\mathfrak g_{i(mod\,n)}, \lambda y - \varphi(y)] = 0$, and hence 
$\varphi(y) = \lambda y$ for any $y \in \mathfrak g_{0(mod\,n)}^\prime$. 
Moreover, setting in (\ref{eq-m}) $i=j=k=0$, we get that the restriction of 
$\varphi$ on $\mathfrak g_{0(mod\,n)}$ is a Hom-Lie structure on 
$\mathfrak g_{0(mod\,n)}$. Since $\mathfrak g_{0(mod\,n)}$ is a simple Lie 
algebra (see table in \cite[p.~129]{Kac}), Theorem \ref{th-s} implies that the 
restriction of $\varphi$ on $\mathfrak g_{0(mod\,n)}$ coincides with 
$\lambda \id_{\mathfrak g_{0(mod\,n)}}$ in all cases except of 
$\mathfrak g_{0(mod\,n)} \simeq \Sl(2)$. In the latter case, a quick glance at
the matrix in the proof of Proposition \ref{prop-sl2} reveals that a Hom-Lie 
structure on $\Sl(2)$ coinciding with the multiple of the identity map on a 
subspace of codimension $1$, has to coincide with the same multiple on the whole
algebra. Therefore, in all cases we have $\varphi(y) = \lambda y$ for any 
$y \in \mathfrak g_{0(mod\,n)}$. But then (\ref{eq-i}) implies that 
$\lambda(\cdot)$ vanishes on the whole $\mathfrak g_{0(mod\,n)}$. The equality
(\ref{eq-m}) implies now that $\varphi$ is a Hom-Lie structure on the whole
$\mathfrak g$. Since $\mathfrak g$ is not isomorphic to $\Sl(2)$ (there are no
nontrivial cyclic gradings in this case; see \cite[Chapter 8]{Kac} for an 
explicit description of all possible cases), 
$\varphi = \lambda \id_{\mathfrak g}$, and $\Phi = \lambda \id_{\mathscr L}$.
\end{proof}

Note that, unlike in the untwisted case, we employed here from the very 
beginning the specific structure of affine Kac-Moody algebras, without 
considering first a more general case of extended current Lie algebras, like in 
Proposition \ref{th-extd}. It is definitely possible to get a ``twisted'' analog
of Proposition \ref{th-extd} by considering a graded Lie algebra of the form
$\bigoplus_{g\in G} \big(L_g \otimes A_g\big)$, where 
$L = \bigoplus_{g \in G} L_g$ and $A = \bigoplus_{g \in G} A_g$ are Lie and
associative commutative algebra, respectively, graded by the same abelian
group $G$, and its derivation acting on the second tensor factors, subject to 
appropriate conditions. This, however, will lead to quite cumbersome statements
with even more cumbersome proofs, seemingly of much less independent interest 
than Proposition \ref{th-extd}, so we have chosen to treat the twisted 
Kac--Moody case directly.

We arrive now at the main result of the paper:

\begin{theorem}\label{th-m}
The space of Hom-Lie structures on an affine Kac--Moody algebra is linearly 
spanned by central Hom-Lie structures and the identity map.
\end{theorem}

\begin{proof}
The nontwisted case is covered by Proposition \ref{prop-km}, and the twisted 
case follows from Proposition \ref{prop-tw} by the same reasonings about Hom-Lie
structures on the central extension as in the proof of 
Proposition \ref{prop-km}.
\end{proof}

Theorem \ref{th-m} implies that if 
$$
\widehat{\mathfrak g} = \mathscr L(\mathfrak g,n) \oplus 
\big(\mathbb C\id_{\mathfrak g} \otimes\> t \frac{\dcobound}{\dcobound t}\big) 
\oplus \mathbb Cz
$$
is an affine Kac--Moody algebra -- untwisted (for $n=1$) or twisted (for 
$n=2,3$) -- then, as a $\widehat{\mathfrak g}$-module, 
$\HomLie (\widehat{\mathfrak g})$ is isomorphic to the direct sum of the dual 
module $\widehat{\mathfrak g}^*$ (corresponding to Hom-Lie structures of the 
form $\widehat{\mathfrak g} \to \mathbb Cz$),
and the one-dimensional trivial module (corresponding to the scalar multiples of
the identity map).

Another easy consequence is the following

\begin{corollary}
The set of multiplicative Hom-Lie structures on an affine Kac--Moody algebra
$\widehat{\mathfrak g}$ coincides with 
\begin{equation*}
\set{\id_{\widehat{\mathfrak g}} + \lambda \beta}{\lambda \in \mathbb C} \cup
\set{\lambda\beta}{\lambda \in \mathbb C} ,
\end{equation*}
where the linear map $\beta$ sends $\id_{\mathfrak g} \otimes t \frac{\dcobound}{\dcobound t}$ to $z$, and is zero
on $\mathscr L(\mathfrak g,n) \oplus \mathbb Cz$.
\end{corollary}

\begin{proof}
By Theorem \ref{th-m}, any Hom-Lie structure on $\widehat{\mathfrak g}$ is of 
the form $\mu \id_{\widehat{\mathfrak g}} + \psi$ for some linear map 
$\psi: \widehat{\mathfrak g} \to \mathbb C z$, and $\mu \in \mathbb C$. Then the
multiplicativity condition (\ref{eq-hom}) yields 
$\psi|_{\mathscr L(\mathfrak g, n) \oplus \mathbb C z} = 
(\mu^2 - \mu)\id_{\mathscr L(\mathfrak g, n) \oplus \mathbb C z}$, and hence 
$\psi(\mathscr L(\mathfrak g, n)) = 0$, and either $\mu = 1$ and $\psi(z) = z$,
or $\mu = 0$ and $\psi(z) = 0$. As $\psi$ can map 
$\id_{\mathfrak g} \otimes t \frac{\dcobound}{\dcobound t}$ to an arbitrary 
multiple of $z$, the statement of the corollary follows.
\end{proof}

Also, it is trivial to derive from here that the set of regular Hom-Lie 
structures on $\widehat{\mathfrak g}$ coincides with 
$\set{\id_{\widehat{\mathfrak g}} + \lambda \beta}{\lambda \in \mathbb C}$. The
latter statement can be also derived in an alternative way, by looking first at
automorphisms of $\widehat{\mathfrak g}$ (there is a wast literature devoted to
description of automorphisms of affine Kac--Moody algebras, see, for example,
\cite[\S 4]{kac-wang} and references therein), and singling out those which 
satisfy the Hom-Jacobi identity.

\section{Jordan algebras and $\delta$-derivations}\label{sec-jord}

It was observed in \cite{chin2} and \cite{xie-liu} that for all the cases
computed so far, Hom-Lie structures on a given Lie algebra $L$ are closed with
respect to Jordan product. That is, if $\varphi, \psi \in \HomLie(L)$, then
$$
\frac 12 (\varphi \circ \psi + \psi \circ \varphi) \in \HomLie(L) ,
$$
so $\HomLie(L)$ forms a Jordan subalgebra of the Jordan algebra $\End(L)^{(+)}$.
Thus an interesting question arises: to which degree this is a common 
phenomenon, and which Jordan algebras can be realized in this way?

Some Lie algebras making appearance in this paper satisfy this property, and 
some are not. For example, for current Lie algebras of the form 
$\mathfrak g \otimes A$, Hom-Lie structures form a Jordan algebra. Indeed,
it follows from (\ref{eq-la}) that 
$\HomLie (\mathfrak g \otimes A) \simeq \HomLie(\mathfrak g) \otimes A$. As
$\HomLie(\mathfrak g)$ is a Jordan algebra, $\HomLie(\mathfrak g) \otimes A$ is
a ``current'' Jordan algebra (by Proposition \ref{prop-sl2} and 
Theorem \ref{th-s}, nontrivial only in the case of $\mathfrak g = \Sl(2)$, and 
in all other cases merely isomorphic to the associative commutative algebra 
$A$).

Hom-Lie structures on affine Kac--Moody Lie algebras, like for any Lie algebra 
whose Hom-Lie structures are linearly spanned by the identity map and central 
Hom-Lie structures, are closed already with respect to composition, so they form
an associative (and hence Jordan) algebra.

On the other hand, it is possible to find examples of current Lie algebras 
$L \otimes A$ for which Hom-Lie structures are not closed with respect to 
Jordan product. For example, assume that $A$ contains a unit $1$, and, following
again the formula (\ref{eq-la}), consider the Jordan product of two Hom-Lie 
structures of the form $\varphi \otimes 1$ and $\psi \otimes \alpha$, where
$\varphi \in \HomLie(L)$, $\psi \in \HomtwoNilp(L)$, and $\alpha \in \End(A)$:
$$
\frac 12 (\varphi \circ \psi + \psi \circ \varphi ) \otimes \alpha .
$$

Obviously, $\psi \circ \varphi \in \HomtwoNilp(L)$, and hence
$(\psi \circ \varphi) \otimes \alpha \in \HomLie(L \otimes A)$, so the question 
reduces to whether $(\varphi \circ \psi) \otimes \alpha$ is a Hom-Lie structure
on $L \otimes A$, i.e., whether
\begin{equation}\label{eq-zero}
  [[x,y],\varphi(\psi(z))] \otimes ab\alpha(c) 
+ [[z,x],\varphi(\psi(y))] \otimes ca\alpha(b)
+ [[y,z],\varphi(\psi(x))] \otimes bc\alpha(a) = 0  
\end{equation}
holds for any $x,y,z \in L$ and $a,b,c \in A$.

Taking $A = K[t]$ (or, for that matter, any $\mathbb Z$-graded algebra with 
sufficient number of nonzero graded components), and monomials $a,b,c \in K[t]$
of different degrees, it is always possible to manufacture $\alpha$ in such a 
way that the degrees of polynomials $ab\alpha(c)$, $ca\alpha(b)$, $bc\alpha(a)$
will be different, and hence the equality (\ref{eq-zero}) reduces to 
$[[L,L],\varphi(\psi(L))] = 0$, i.e., to the question whether 
$\varphi \circ \psi \in \HomtwoNilp(L)$. Now take 
$L = \langle x,y \vertbar [x,y]=x \rangle$, the two-dimensional nonabelian Lie algebra. Then 
$\HomLie(L) = \End(L)$, 
$\HomtwoNilp(L) = \set{\psi \in \End(L)}{\psi(L) \subseteq \langle x \rangle}$,
and taking 
$$
\varphi: \>
\begin{aligned}
x &\mapsto y \\
y &\mapsto 0
\end{aligned}
\hskip 40pt
\psi: \>
\begin{aligned}
x &\mapsto x \\
y &\mapsto 0
\end{aligned}
$$
we have $\varphi \circ \psi = \varphi \notin \HomtwoNilp(L)$.

Another, somewhat related to the above, question concerns $\delta$-derivations.
It is proved in \cite[Proof of Theorem 1]{filippov-1} that for any 
$\delta \ne 0,1$, any $\delta$-derivation of a Lie algebra to itself is a 
Hom-Lie structure. On the other hand, one may easily provide examples of Lie 
algebras for which not every Hom-Lie structure is a $\delta$-derivation. In a 
somewhat different vein, $\delta$-derivation can be trivially put in the 
framework of Hom-Lie structures by extending the underlying Lie algebra: namely,
for any Lie algebra $L$ and $\delta$-derivation $D$ of $L$, $\delta \ne 0$, the
semidirect sum $L \oplus KD$ admits a Hom-Lie structure $\alpha$ by letting 
$\alpha$ act on $L$ identically, and $\alpha(D) = \frac{1}{\delta} D$. We leave
a thorough exploration of relationships between $\delta$-derivations and Hom-Lie structures to the 
future.

Finally, yet another interesting question is how Hom-Lie structures behave under
deformations of Lie algebras. In the case where Hom-Lie structures form a Jordan
algebra, one may ask then how this fits the deformation theory of Jordan 
algebras.

\section*{Acknowledgements}

This work was started during the visit of the second author to Universit\'e de
Haute-Alsace. GAP \cite{gap} was used to confirm some of the computations.

\end{document}